\documentclass[a4paper,11pt,reqno]{amsart}
\usepackage[utf8]{inputenc}
\usepackage[T1]{fontenc}

\usepackage{amsmath}
\usepackage{amssymb}
\usepackage{amsthm}				
\usepackage{mathtools}				
\usepackage{mathrsfs}				
\usepackage{dsfont}				
\usepackage{subcaption}
\usepackage{cite}
\usepackage[hidelinks]{hyperref}
\usepackage{paralist}

\newcommand{\Ex}{\mathbb{E}}
\newcommand{\Var}{\mathbb{V}}

\newcommand{\mc}{\mathcal}
\newcommand{\msc}{\mathscr}
\newcommand{\br}{\mathsf{b}}
\newcommand{\betw}{\mathsf{i}}
\newcommand{\ext}{\mathsf{e}}

\newcommand{\permset}{\underline{\pi}}
\newcommand{\ssq}{\subseteq}

\mathtoolsset{showonlyrefs,showmanualtags}

\newtheorem{lemma}{Lemma}
\newtheorem{theorem}[lemma]{Theorem}
\newtheorem{proposition}[lemma]{Proposition}
\theoremstyle{definition}

\newtheorem{conjecture}[lemma]{Conjecture}
\newtheorem{remark}[lemma]{Remark}

\title{A bijection for the evolution of \textit{B}-trees}
\author[F. Burghart]{Fabian Burghart}
\address{Department of Mathematics and Computer Science, Eindhoven University of Technology, 5612AE Eindhoven, The Netherlands}
\email{f.burghart@tue.nl}
\author[S. Wagner]{Stephan Wagner}
\address{Institute of Discrete Mathematics, TU Graz, Steyrergasse 30, 8010 Graz, Austria \and Department of Mathematics, Uppsala University, Box 480, 751 06 Uppsala, Sweden}
\email{stephan.wagner@tugraz.at}
\date{}
\keywords{\textit{B}-trees, histories, increasing trees, bijection, asymptotic enumeration, tree statistics}
\subjclass[2020]{05A19 (Primary); 68P10, 60C05, 05A05, 05A16 (Secondary)}

\begin{document}

\begin{abstract}
A $B$-tree is a type of search tree where every node (except possibly for the root) contains between $m$ and $2m$ keys for some positive integer $m$, and all leaves have the same distance to the root. We study sequences of $B$-trees that can arise from successively inserting keys, and in particular present a bijection between such sequences (which we call histories) and a special type of increasing trees. We describe the set of permutations for the keys that belong to a given history, and also show how to use this bijection to analyse statistics associated with $B$-trees.
\end{abstract}

\maketitle

\section{Introduction and main results}

$B$-trees, since their inception in \cite{BM72}, have become a popular data structure. Regarding their mathematical analysis, there were some early results by Yao \cite{yao78} and Odlyzko \cite{odlyzko} for the special case of 2-3-trees, but despite Knuth posing a natural open question in \cite{chvatal72}, progress has been scarce. Perhaps most notable is the approach using P\'{o}lya urns as in \cite{BP85,BY95,AFP88,CGPT16}, which yielded results especially for the fringe analysis of $B$-trees. In this paper, we propose a novel way of investigating $B$-trees, by focussing on what we call histories. 

\subsection{\textit{B}-trees and their insertion algorithm}

By a search tree, we mean a rooted plane tree whose nodes contain \emph{keys}, which we think of as pairwise distinct real numbers, in such a way that (1) the keys are stored in increasing order from left to right (including within a single node), and (2) every non-leaf node containing $k$ keys has exactly $k+1$ children, where we think of the $i$-th child as being attached between the $(i-1)$-th and the $i$-th key of its parent node. For $i=1$ we interpret this as being attached to the left of the first key, and analogously for $i=k+1$, the child is attached to the right of the last key in the node. Note that we explicitly allow leaves to contain keys, and will refer to the intervals between consecutive keys in a leaf as \emph{gaps}; thus we do not follow the convention of \cite{knuth} where the leaves really take the place of our gaps, and therefore cannot contain keys. 

Let $m\geq 1$. Following Knuth \cite[Section 6.2.4]{knuth}, a \emph{$B$-tree} of order $2m+1$ is a search tree satisfying the following properties: Every node contains at least $m$ and at most $2m$ keys, except for the root which contains at least $1$ and at most $2m$ keys. Moreover, the tree is balanced in the sense that all leaves have the same distance to the root. We remark that some authors (e.g. \cite{BM72}) refer to such a tree as a $B$-tree of order $m$ instead. 

$B$-trees can be constructed via the following \emph{insertion algorithm}: Given a $B$-tree and a key that is not already stored in the tree, place the key in the appropriate leaf and the appropriate position within the keys of the leaf. If, after this placement, the leaf still contains at most $2m$ keys, then we are done. Otherwise, we split the node containing $2m+1$ keys by moving the median key up into the parent node and grouping the lowest $m$ keys and the largest $m$ keys each in their own node. By doing this, it might now happen that the parent node contains $2m+1$ keys, in which case we again split it into two nodes of $m$ vertices and move the median key (of the parent node) up. This process may propagate all the way along the path from the leaf where we inserted the key to the root vertex, in which case we create a new root vertex above the old root, containing only a single key (the one that was the median among the $2m+1$ keys of the old root), and split the old root in two. Note that the latter case of splitting the root is the only situation in which the height of the $B$-tree can increase. 

For the purpose of this article, we are interested in $B$-trees up to isomorphism of rooted plane trees. Equivalently, we can represent an isomorphism class by replacing all keys by dots, as in Figure~\ref{fig:bijection}(left). For brevity's sake, we will henceforth use $B$-tree to mean such an isomorphism class. An alternative way to think about these isomorphism classes is to fix the keys instead, e.g. by saying the keys are the set $\{1,\dots,n\}$---the disadvantage of this approach being that inserting another key means having to reassign the values of some of the old keys. Nonetheless, we will make use of both of these representations.

\subsection{Main results}

Let $T_n$ be a $B$-tree of order $2m+1$ containing $n$ keys. A \emph{history} of $T_n$ is a finite sequence $(T_1,\dots,T_n)$ of $B$-trees of order $2m+1$ such that for all $i=2,\dots,n$, the tree $T_i$ is obtained from $T_{i-1}$ through inserting a single key using the insertion algorithm outlined above. In particular, $T_i$ contains $i$ keys. We denote by $\msc H_m(T_n)$ the set of all histories of $T_n$, and by $\msc H_m(n)$ the set of all histories of any $B$-tree of order $2m+1$ with $n$ keys. In other words, $\msc H_m(n)=\bigcup_{T_n} \msc H_m(T_n)$, where the union is taken over all (non-isomorphic) $B$-trees of order $2m+1$ with $n$ keys. 

We can now state our main result:

\begin{theorem}\label{thm:main}
 Let $n,m\geq 1$. There is a bijection between $\msc H_m(n)$ and the set of all trees $H_n$ satisfying the following properties:
 \begin{enumerate}[(i)]
  \item $H_n$ is a rooted plane tree on $n$ vertices, labelled by $\{1,\dots,n\}$, such that along each path from the root to a leaf, the labels are increasing. 
  \item The vertices of $H_n$ at heights $2m, 3m+1, 4m+2, \dots$ have up to two children, all other vertices have at most one child. 
 \end{enumerate}
\end{theorem}

We will call trees $H_n$ satisfying properties (i) and (ii) in the theorem $(2m+1)$-\emph{historic} (or just \emph{historic}, if it is not ambiguous) in the interest of brevity. Given a historic tree $H$ on $n$ vertices, it will be useful throughout to consider all potential positions for attaching a vertex $n+1$ that lead to another historic tree. We think of these positions as \emph{external vertices}, and call the vertices in $H$ \emph{internal} to tell them apart. We also write $\bar H$ to denote $H$ together with the external vertices. Furthermore, we call the internal vertices at height $2m, 3m+1, 4m+2,\dots$ \emph{branchings} (irrespective of how many internal children they have). 

\begin{proposition}\label{prop:bijection}
 Let $H_n$ be the $(2m+1)$-historic tree corresponding to a history $(T_1,\dots,T_n)$ of $B$-trees of order $2m+1$ under the bijection in Theorem~\ref{thm:main}. Then, the following holds: 
 \begin{enumerate}[(i)]
  \item For any $n\geq 1$, the number of external vertices of $H_n$ equals the number of leaves of $T_n$. 
  \item For any $n\geq 1$, the number of branchings in $H_n$ equals the number of keys in $T_n$ that are not stored in leaves. 
  \item Let $n\geq 2m+1$. Consider the $i$-th external vertex $v$ of $H_n$ from the left, and let $s$ be the number of internal vertices in $H_n$ strictly between $v$ and the closest branching above~$v$. Then, the $i$-th leaf of $T_n$ from the left contains exactly $m+s$ keys. 
 \end{enumerate}
\end{proposition}

We dedicate Section~\ref{section:bijection} to the proof of Theorem~\ref{thm:main} and Proposition~\ref{prop:bijection}. That section will also contain the description of the bijection. In Section~\ref{section:perm}, we exhibit a recursive construction of $\permset(H_n)$, the set of all permutations $\pi\in S_n$ that, when used as key sequence for a $B$-tree, lead to the history described by the historic tree $H_n$. As part of this description, we obtain the following result:

\begin{proposition}\label{prop:permnumber}
 Let $H_n$ be a $(2m+1)$-historic tree having $b\geq 1$ branchings. Let $s_1,\dots,s_{b+1}$ be the number of internal vertices in $H_n$ strictly between the $i$-th external vertex and its closest branching. Then
 \begin{equation}\label{eq:permnumber}
  |\permset(H_n)|=\left(\frac{(2m+1)!}{(m!)^2}\right)^b\cdot \prod_{i=1}^{b+1} (m+s_i)!.
 \end{equation}
\end{proposition}

This formula is somewhat reminiscent of the classical hook length formula, see e.g.~\cite[Section 5.1.4, Exercise 20]{knuth}: the number of increasing labellings of a tree with $n$ vertices is given by
\begin{equation*}
n! \prod_v \frac{1}{N_v},
\end{equation*}
where the product is over all vertices and $N_v$ is the number of vertices in the subtree consisting of $v$ and all its descendants.

\begin{remark}
 It is possible to consider $B$-trees of order $2m$ as well, where a node splits whenever it is assigned $2m$ keys. In that case, the smallest $m-1$ keys end up in the left node, the $m$-th key is pushed into the parent node, and the largest $m$ keys end up in the right node. It is still possible to define suitable $(2m)$-historic trees, but the distance between a branching and the next branching below will depend on whether we go to the left or to the right in $H_n$. 
\end{remark}

\section{The bijection}\label{section:bijection}

\begin{figure}
 \includegraphics[width=\textwidth]{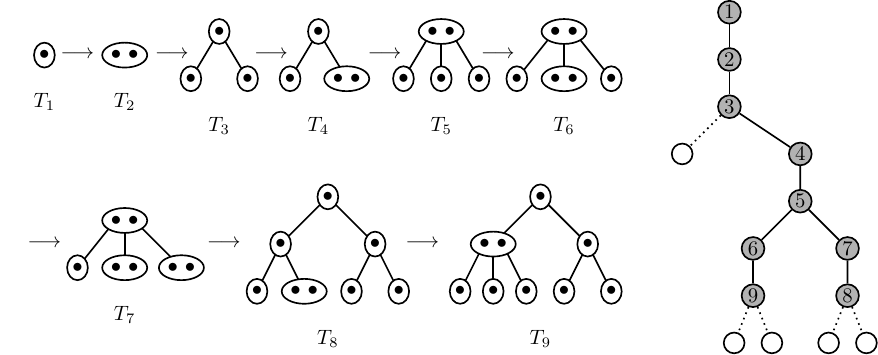}
 \caption{A history of $B$-trees of order $2m+1=3$ on the left, with the corresponding historic tree $H_9$ shown on the right. The external vertices of $H_9$ are shown in white, and are connected by dotted lines. The vertices $3,5,8$, and $9$ are the branchings of $H_9$.}\label{fig:bijection} 
\end{figure}

The purpose of this section is to prove Theorem~\ref{thm:main}. 

We begin by describing the bijection; see Figure~\ref{fig:bijection} for an example. If $n=1$, there is only one $B$-tree and only one corresponding $H_1$. For an arbitrary history $(T_1,\dots,T_n)\in \msc H_m(n)$, construct the corresponding $H_n$ as follows: Assume we already constructed $H_k$ corresponding to the history $(T_1,\dots,T_k)\in \msc H_m(k)$ for some $1\leq k< n$. Then $T_{k+1}$ is obtained from $T_k$ by inserting a single new key. If this insertion takes place in the $i$-th leaf (counted from left to right) of $T_k$ before accounting for possible splits, then we attach the vertex labelled $k+1$ to $H_k$ at the $i$-th external vertex of $H_k$ (counted from left to right). This gives $H_{k+1}$, and inductively, $H_n$. 
 
Conversely, given some historic $H_n$, we can construct trees $H_1,\dots,H_{n-1}$ such that $H_k$ is the subtree consisting of the vertices with label $\leq k$. Suppose that the vertex $k+1$ is attached to $H_k\ssq H_{k+1}$ in the $i$-th external vertex of $H_k$, and suppose we have already constructed the history $(T_1,\dots,T_k)$ corresponding to $H_k$. Then we can extend this history to the one corresponding to $H_{k+1}$ by inserting a key into the $i$-th leaf from the left of $T_k$, and let $T_{k+1}$ be the $B$-tree obtained by this (possibly after performing the necessary splits). 
 
It is clear from the description that this gives inverse maps between $\msc H_m(n)$ and $(2m+1)$-historic trees on $n$ vertices, provided the constructions are at all well-defined. This is the case if the number of external vertices on $H_k$ equals the number of leaves of $T_k$, which is exactly claim (i) in Proposition~\ref{prop:bijection}. Thus we proceed by proving Proposition~\ref{prop:bijection}, which will imply Theorem~\ref{thm:main}. 

\begin{proof}[Proof of Proposition~\ref{prop:bijection}]
 We first note that (i) is equivalent to (ii). Indeed, since all non-branchings have outdegree 1 in $\overline{H_n}$, and the branchings have outdegree exactly 2, the number of branchings is one less than the number of external vertices. Similarly, it is a simple consequence of the insertion algorithm for $B$-trees that every key that gets moved out of a leaf by a split increases the number of leaves by one, so that the number of keys not stored in leaves is one less than the number of leaves. 
 
 Next, we observe that (i) holds for $n\leq 2m$. This is the case since any $B$-tree of order $2m+1$ for those values of $n$ only has a single node (which is simultaneously root and leaf), and all vertices with these labels in $H_n$ necessarily have outdegree 1. We proceed by induction on $n$. 
 
 For $n=2m+1$, we see the first split in $T_n$, leading to a root node containing a single key, and two leaves containing $m$ keys each. For $H_n$, we have now reached height $2m$, and thus have two external vertices -- these are the children of a branching in $\overline{H_n}$, thus there are no internal vertices between them and the branching. This establishes both (i) and (iii) for $n=2m+1$. 
 
 Now assume that (i) and (iii) hold for some $n\geq 2m+1$, and that we obtain $T_{n+1}$ from $T_n$ by adding a key to the $i$-th leaf, which held $m+s$ keys in $T_n$. We distinguish two cases for $s$:
 
 For $0\leq s\leq m-1$, we end up with $m+s+1$ keys in the $i$-th leaf of $T_{n+1}$, and the number of leaves does not change. For $H_n$, we need to append the vertex labelled $n+1$ in place of the $i$-th external vertex. Denote by $w$ the closest branching to $n+1$ (i.e., the most recent predecessor that is a branching; such a vertex exists since $n\geq 2m+1$). By assumption, there are exactly $s$ vertices strictly between $w$ and $n+1$, so $n+1$ is not a branching, and only has a single external child (and the path from that external vertex to $w$ is one vertex longer). Thus, properties (i) and (iii) hold for $n+1$. 
 
 If, on the other hand, $s=m$, then adding the key splits the $i$-th leaf; producing two leaves in its stead that each hold $m$ keys. For $H_n$, we denote again by $w$ the closest branching to the $i$-th external vertex which becomes the position of the new vertex $n+1$.  Invoking (iii) for $H_n$ shows that there are $m$ vertices between $n+1$ and $w$, so $n+1$ is another branching and therefore has two new external vertices as children, replacing the old one. The closest branching to the new external vertices is now $n+1$, and there are no internal vertices strictly between them, which again corresponds to the number of keys in the new vertices. This shows that splits in the $B$-tree correspond to branchings in the historic tree and asserts (i) and (iii) for $n+1$, finishing the induction argument.  
\end{proof}

\section{The permutations associated with a history}\label{section:perm}

Let $T$ be a $B$-tree of order $2m+1$, containing $n$ keys. We denote by $\permset(T)$ the set of all permutations $\pi\in S_n$ that, when used as a key sequence for the insertion algorithm, yield the tree $T$. The aim of this section is to give a recursive description of $\permset(T)$ in terms of the ``trimmed'' tree $T^{(1)}$ that is obtained from $T$ by deleting all leaves. For this purpose, write $n_1$ for the number of keys stored in $T^{(1)}$. We will rely on the following observation:

Consider a history $(T_1,\dots,T_n=T)$. Let $i_1<i_2<\dots<i_{n_1}$ be those $i$ where $T_i$ was obtained from $T_{i-1}$ by inserting a key that led to a split (note that this is consistent with the indexing). Then $\big(T^{(1)}_{i_1},\dots,T^{(1)}_{i_{n_1}}\big)$ is a valid history of $T^{(1)}$. We remark that this is a purely combinatorial statement: If we instead looked at i.i.d. keys sampled from some continuous probability distribution, then the processes $(T_n)_{n\geq 1}$ and $\big(T^{(1)}_{i_n}\big)_{n\geq 1}$ would be quite different!

As a consequence, suppose we are given a permutation $\pi\in\permset(T)$. This $\pi$ produces a history $(T_1,\dots,T_n)$. Moreover, keeping track of the actual keys, we obtain a sequence $K_{i_1},\dots,K_{i_{n_1}}$ of those keys that ascend above the leaves at times $i_1,\dots,i_{n_1}$. Forgetting about their actual values and only keeping track of the relative size of the $K_{i_j}$ then produces a new permutation $\pi^{(1)}\in\permset(T^{(1)})$, where, moreover, $\pi^{(1)}$ produces the history $\big(T^{(1)}_{i_1},\dots,T^{(1)}_{i_{n_1}}\big)$. This defines a map $\Psi_T:\permset(T) \to \permset(T^{(1)}), \pi\mapsto \pi^{(1)}$, and our goal will be to invert this: Given a $\pi^{(1)}$, we want to find all $\pi\in S_n$ that lead to such $\pi^{(1)}$. 

This inversion will come in the form of a 3-step algorithm, described in detail below, in Section~\ref{subsection:alg}. However, we will give a high-level overview now:
\begin{enumerate}
 \item In the first step, we start from a given $\pi^{(1)}\in\permset(T^{(1)})$, and lift it to a sequence $(K_{i_1},\dots,K_{i_{n_1}})$ as above. 
 \item In the second step, we use $\pi^{(1)}$ and Proposition~\ref{prop:bijection} to construct an acyclic digraph $G=G(T,\pi^{(1)})$. Lemma~\ref{lemma:digraph} states that the set of topological labellings of $G$ corresponds bijectively to the set of historic trees of $T$ that produce $\pi^{(1)}$. 
 \item Therefore, in the third step, we can fix a historic tree $H$ obtained from step 2, and restrict our attention to $\permset(H)$. The algorithm given will produce an arbitrary element of $\permset(H)$ after making a sequence of choices; different choices will lead to different permutations, and going over all permitted choices produces the entire set, see Lemma~\ref{lemma:step3}. In more concrete terms, we start step 3 with an ``empty'' permutation consisting of $n$ blank symbols, and by recursively comparing it against $K_{i_1},\dots, K_{i_{n_1}}$ and $H$ we will replace the blanks by entries from $\{1,\dots,n\}$. 
\end{enumerate}

\subsection{Preparatory lemmas}

To ensure well-definedness at a later point (Lemma~\ref{lemma:digraph}), we need the following lemma:
\begin{lemma}\label{lemma:psihat}
 There is a well-defined map $\hat\Psi_T:\msc{H}_m(T) \to \permset(T^{(1)})$ that assigns to a history $(T_1,\dots,T_n=T)$ the $\pi^{(1)}$ constructed above, where $\pi\in\permset(T)$ is any permutation producing the history. 
\end{lemma}
\begin{proof}
 We will assume that the keys in $T_j$ are exactly $1,\dots,j$ (labelled from left to right, since $T_j$ is a search tree) for $1\leq j\leq n$, and relabel them accordingly whenever we insert a new key. We show inductively that we can (a) determine uniquely which key moved up from the leaves at the times $i_1,i_2,\dots$ and (b) keep track of how the keys in $T^{(1)}_{j-1}$ change as we go to $T^{(1)}_j$. For $T_1,\dots,T_{2m}$, there is nothing to show. In $T_{2m+1}$, we know that the unique key in the root node has label $m+1$. Suppose we have verified (a) and (b) for some $j\geq 2m+1$. If $j+1$ is not one of $i_1,i_2,\dots,i_{n_1}$ then no splits happen, and comparing $T_{j+1}$ with $T_j$ reveals which leaf grew by one. All keys in $T^{(1)}_j$ to the right of that leaf are increased by 1 for $T^{(1)}_{j+1}$, all other keys in $T^{(1)}_j$ remain the same. If on the other hand $j+1$ is one of $i_1,i_2,\dots,i_{n_1}$ then comparing $T_{j+1}$ and $T_j$ reveals which leaf of $T_j$ split. As before, all keys in $T^{(1)}_j$ to the right of that leaf are increased by 1 for $T^{(1)}_{j+1}$, all other keys in $T^{(1)}_j$ remain the same. Moreover, let $K$ be the largest key in $T^{(1)}_j$ to the left of the splitting leaf. Then the new key introduced to $T^{(1)}_{j+1}$ will be $K+m+1$, and it will be the unique key in $T_{j+1}$ that is placed between the two leaves coming from the split leaf. 
 
 Thus, only from the history of $T_n$ we can keep track of which keys were introduced to $T^{(1)}_j$ in which order, which yields $K_{i_1},\dots,K_{i_{n_1}}$ after updating all the keys and thus $\pi^{(1)}$.
\end{proof}

We also note the following simple fact about the bijection from Theorem~\ref{thm:main}:
\begin{lemma}\label{lemma:Hfact}
 Let $H_n$ be the historic tree for $(T_1,\dots,T_n)$. Suppose that vertex $i$ of a historic tree $H$ is a branching, and suppose that the key that is pushed upwards from the splitting leaf at that time is $K_i\in\{1,\ldots,n\}$ in $T_n$. Let $j\in\{i+1,\dots,n\}$ be another vertex of $H$, and let $k_j$ denote the key added at time $j$ in the history. Then, $k_j> K_i$ if and only if $j$ is positioned to the right of $i$ in $H$ (not necessarily as a descendant of $i$), and $k_j< K_i$ otherwise. Moreover, if $j>i$ is another branching of $H$, then also $K_j>K_i$ if and only if $j$ is to the right of $i$, and $K_j<K_i$ otherwise. 
\end{lemma}
\begin{proof}
 This follows from the observation that after the split pushes $K_i$ upwards, the leaves of the $B$-tree can be partitioned into those containing keys $<K_i$, which are therefore to the left, and those containing keys $>K_i$, which are further to the right, as well as from the description of the bijection given in Section~\ref{section:bijection}.
\end{proof}

\subsection{The algorithm}\label{subsection:alg}

We now turn our attention to the promised ``inverse'' of $\Psi_T$. Denote by $h$ the height of $T$. 

For $h=0$, the tree $T$ only consists of the root node, and then $\permset(T)=S_n$. For $h>0$, suppose we know $\permset(T^{(1)})$.

\emph{Step 1:} By performing an in-order traversal of the keys in $T$, we can see which of the numbers $1,\dots,n$ correspond to the keys in $T^{(1)}$. In other words, in-order traversal gives a monotone injection $\iota:\{1,\dots,n_1\} \hookrightarrow \{1,\dots,n\}$, by sending $i$ to the $j$ that is the $i$-th key from the left among those not in a leaf node of $T$. This injection in turn allows us to write any $\pi^{(1)}\in\permset(T^{(1)})$ as a sequence $\pi_\iota=\big(K_{i_1},\dots,K_{i_{n_1}}\big)$. 

\emph{Step 2:} We construct a rooted digraph $G=G(T,\pi^{(1)})$ in the following way: First, construct a binary search tree from $\pi^{(1)}$. Then, subdivide the edges (and move the root up) in such a fashion that the nodes of the binary search tree become the branchings of a historic tree and append extra vertices to match with the leaves of $T$, according to Proposition~\ref{prop:bijection}(iii). We then direct all edges away from the root, and consider the directed path $\pi^{(1)}(1)\longrightarrow \dots\longrightarrow \pi^{(1)}(n_1)$. Merge this path into the (mostly empty) historic tree by identifying the vertex $\pi^{(1)}(i)$ in the path with the vertex containing $\pi^{(1)}(i)$ in the tree, for all $i=1,\dots,n_1$. For bookkeeping, we colour the edges coming from the path red, and the edges from the tree black. Finally, delete all labels/keys from the resulting digraph $G$. 

\begin{lemma}\label{lemma:digraph}
 The digraph $G=G(T,\pi^{(1)})$ constructed in this fashion is acyclic. Furthermore, any topological labelling of $G$ (that is, any labelling such that all edges point towards the higher label) induces a historic tree $H$ for $T$ on the black edges. Such $H$ corresponds to those histories of $T$ that are obtained by $\pi\in S_n$ such that $\pi^{(1)}$ is the associated history of $T^{(1)}$. In other words, we have
 \[
  \{\text{topological labellings of }G(T,\pi^{(1)}) \} \stackrel{1:1}{\longleftrightarrow} \hat\Psi^{-1}(\pi^{(1)}),
 \]
 where the bijection is the one from Theorem~\ref{thm:main} after removing the red edges from $G$.
\end{lemma}

\emph{Step 3:} It remains to give the actual description of $\permset(T)$. Specifically, writing $\permset(H)$ for the set of $\pi\in S_n$ that produce the history encoded by the historic tree $H$, we can pick an $H$ coming from a topological labelling of $G(T,\pi^{(1)})$ and describe the corresponding $\permset(H)$. We are given the key sequence $\pi_\iota=(K_{i_1},\dots,K_{i_{n_1}})$ from Step 1, as well as a fixed topological labelling of $G$, with the induced historic tree $H$. 

In what follows, $p$ will be a sequence of distinct integers which is to be determined, thought of as a map onto some range $\mc R$. Furthermore $\mc H$ will be a historic tree on $|\mc R|$ vertices, and $\mc K$ is a subsequence of $\pi_\iota$ containing (in the same order) all those $K_{i_j}$ that appear in $\mc R$. Moreover, we demand that the length of the sequence $\mc K$ equals the number of branchings in $\mc H$. Then, the following recursive procedure constructs all desired $\pi$: 

\emph{Step 3.0:} Initialize $p=\pi$ as a yet undetermined permutation in $S_n$, thus $\mc R=\{1,\dots,n\}$. Further, set $\mc H=H$ and $\mc K=\pi_\iota$. 

\emph{Step 3.1:} If $|\mc R|\leq 2m$, let $p$ be an arbitrary bijection onto $\mc R$. Otherwise, choose an arbitrary position $1\leq j_1\leq 2m+1$ to place $\mc K_1$, the first element of $\mc K$ (i.e., fix $p(j_1)=\mc K_1$), mark $m$ additional positions among the first $2m+1$ of $p$ as \emph{small}, and the remaining $m$ as \emph{large}. For $j>2m+1$, mark the $j$-th entry of the permutation as \emph{small} if the vertex labelled $j$ is positioned to the left of the topmost branching in $H$, and as \emph{large} otherwise. 

\emph{Step 3.2:} Define new undetermined bijections $p^\pm$, where $p^+$ contains all the large positions of $p$, and $p^-$ all the small ones. These bijections will have the ranges $\mc R^+:=\mc R \cap \{\mc K_1+1,\dots,n\}$ and $\mc R^-:=\mc R\cap \{1,\dots,\mc K_1-1\}$, respectively. Moreover, let $\mc K^\pm$ be the subsequences of $\mc K$ containing, in the same order, the entries strictly larger/smaller than $\mc K_1$.   Also split $\mc H$ into $\mc H^\pm$ such that $\mc H^-$ contains the vertices labelling the small positions in $\pi$ and such that $\mc H^-$ below the $m$-th vertex is equal to the left subtree of $\mc H$ from Step 3.1. Construct $\mc H^+$ analogously, then relabel $\mc H^\pm$ with integers from $1,\ldots,|\mc H^\pm|$ while maintaining the relative order. 

\emph{Step 3.3:} Repeat steps 3.1--3.3 for both $(p^\pm,\mc R^\pm,\mc K^\pm,\mc H^\pm)$. 

\begin{lemma}\label{lemma:step3}
 If $H$ comes from a topological ordering of $G(T,\pi^{(1)})$ and $\pi_\iota$ is constructed from $T$ and $\pi^{(1)}$ as in step 1, then the following holds for step 3.2:
 \begin{enumerate}[(i)]
  \item $\mc K^\pm$ consists of those entries of $\pi_\iota$ that are contained in $\mc R^\pm$.
  \item $|\mc H^\pm|=|\mc R^\pm|\geq m$, and each of $\mc R^\pm$ is a set of consecutive integers.
  \item The lengths of $\mc K^\pm$ are equal to the number of branchings in $\mc H^\pm$. 
 \end{enumerate}
 Moreover, the set of permutations constructible with step 3 is $\permset(H)$.  
\end{lemma}

\begin{remark}
 Since we have $\permset(T) = \bigcup_H \permset(H)$, where the union is disjoint and to be taken over all histories leading to $T$, this means we can construct $\permset(T)$ out of $\permset(T^{(1)})$ by performing steps 1--3 for all $\pi^{(1)}$ in $\permset(T^{(1)})$.
\end{remark}

\subsection{An example}\label{subsection:ex}

We give an example to illustrate the procedure: Suppose $m=1$, $n=9$, and consider the permutation $\pi=(6,1,2,4,7,5,9,8,3)$. This permutation produces a $B$-tree $T$ of the form given in Figure~\ref{fig:exT} -- in fact, this permutation gives the history shown in Figure~\ref{fig:bijection}. Thus $T^{(1)}$ contains 4 keys, and $\pi^{(1)}=(1,3,4,2)\in\permset(T^{(1)})$. 

\emph{Step 1:} The in-order traversal of $T$ reveals that the keys in $T^{(1)}$ correspond to the keys $2,4,6,8$ in $T$. Then the injection on the keys is given by $\{1,2,3,4\}\ni i\mapsto 2i\in\{1,\dots,9\}$, and applying this to the entries of $\pi$ gives $\pi_\iota=(2,6,8,4)$. 

\emph{Step 2:} Constructing a binary search tree from $\pi^{(1)}$ gives the one shown in Figure~\ref{fig:exBST} which is then turned into the DAG $G=G(T,\pi^{(1)})$ shown in Figure~\ref{fig:exG} (the remaining labels are there to indicate how it connects to the binary search tree and to $\pi^{(1)}$). This graph has three distinct topological labellings, one of which induces the $H$ depicted in Figure~\ref{fig:exH}. 

\emph{Step 3:} We initialize $p=(\_,\_,\_,\_,\_,\_,\_,\_,\_)$, $\mc R=\{1,\dots,9\}$, $\mc K=(2,6,8,4)$ and  $\mc H=H$. After step 3.1, we have e.g. $p=(\ell,s,2,\ell,\ell,\ell,\ell,\ell,\ell)$, where we write $s$ for a small position, and $\ell$ for a large. Here, the assignment of $\ell,s,2$ to the first 3 positions can be done arbitrarily (but we choose the options that will reconstruct $\pi$ from above), the remainder is given by comparing it to $H$: the vertices labelled $4,5,\dots$ are all positioned to the right of the top-most branching in $H$. This leads to $p^-=(\_)$ with $\mc R^-=\{1\}$, which in the next round of the recursion simply becomes $p=(1)$, and to $p^+=(\_,\_,\_,\_,\_,\_,\_)$ with $\mc R^+=\{3,\dots,9\}$, $\mc K^+=\{6,8,4\}$, and an $\mc H^+=H'$ given by Figure~\ref{fig:exH2}.

In the second round of the recursion (using the ``$+$''-branch, as the other one is trivial), we have e.g. $p=(6,s,\ell,s,\ell,\ell,s)$, where the assignment of $\ell,s,6$ to the first 3 positions can again be done arbitrarily, and the rest is governed by $H'$. This gives $p^\pm=(\_,\_,\_)$ with $\mc R^-=\{3,4,5\}, \mc K^-=(4)$ and $\mc R^+=\{7,8,9\}, \mc K^+=(8)$, respectively. Due to the small number of entries, both $\mc H^\pm$ are given by the unique 3-historic tree on 3 vertices. In the next two rounds of the recursion, the $p^\pm$ will then be filled by arbitrary assignments of the numbers in their range, say $p^+=(7,9,8)$ and $p^-=(4,5,3)$. 

Finally, we can put everything back together by embedding a pair of $p^\pm$ into the previous $p$ according to the assignment of $s$ and $\ell$. Thus, $p^+=(7,9,8)$ and $p^-=(4,5,3)$ together yield $(6,4,7,5,9,8,3)$. This in turn was $p^+$ from the first iteration of step 3, and together with the corresponding $p^-=(1)$, we regain $\pi=(6,1,2,4,7,5,9,8,3)$.  

\begin{figure}
 \centering
 \captionsetup[subfigure]{justification=centering}
 \begin{minipage}[b][][t]{.2\textwidth}
  \begin{subfigure}[t]{\textwidth}
   \centering
   \includegraphics[width=\textwidth]{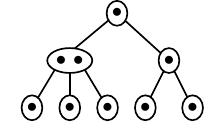}
   \caption{The $B$-tree $T$}
   \label{fig:exT}
  \end{subfigure}
  \\[1em]
  \begin{subfigure}{\textwidth}
   \centering
   \includegraphics[width=0.8\textwidth]{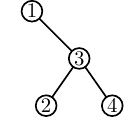}
   \caption{Binary search tree}
   \label{fig:exBST}
  \end{subfigure}
 \end{minipage}
 \hfill
 \begin{minipage}[b]{0.25\textwidth}
  \begin{subfigure}[b]{\textwidth}
   \centering
   \includegraphics[width=\textwidth]{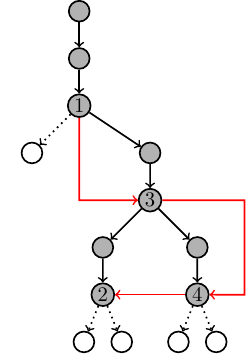}
   \caption{$G(T,\pi^{(1)})$}
   \label{fig:exG}
  \end{subfigure}
 \end{minipage}
 \hfill
 \begin{minipage}[b]{0.25\textwidth} 
  \begin{subfigure}[b]{\textwidth}
   \centering
   \includegraphics[width=\textwidth]{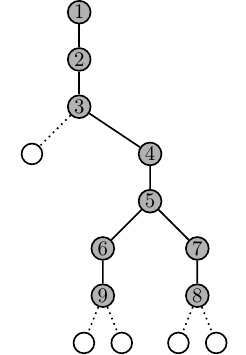}
   \caption{$H$}
   \label{fig:exH}
  \end{subfigure}
 \end{minipage}
 \hfill
 \begin{minipage}[b]{0.2\textwidth}
  \begin{subfigure}[b]{\textwidth}
   \centering
   \includegraphics[width=\textwidth]{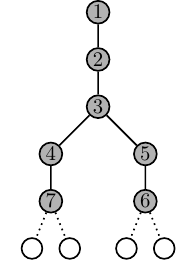}
   \caption{$H'$}
   \label{fig:exH2}
  \end{subfigure}
 \end{minipage}
 \caption{Steps in the algorithm of Section~\ref{section:perm}, as performed in Subsection~\ref{subsection:ex}.}
\end{figure}

\subsection{Proofs}

\begin{proof}[Proof of Lemma~\ref{lemma:digraph}]
 For $G$ to contain a directed cycle, we need two vertices $v_i=\pi^{(1)}(i)$ and $v_j=\pi^{(1)}(j)$ such that $v_i$ is a descendant of $v_j$ in the tree (i.e., according to the black edges), but $v_j$ is a descendant of $v_i$ according to the red edges. However, the latter only means that $i<j$. Accordingly, $\pi^{(1)}(i)$ was the first to be used for the binary search tree's construction, and hence $v_i$ cannot be below $v_j$ in the tree. Thus $G$ is acyclic. 
 
 Trivially, any topological ordering of $G$ yields an increasing labelling for the tree, so the induced $H$ is historic. By construction of $H$, the final tree in the corresponding history $(T_1,\dots,T_n)$ will have the same leaves as $T$ (according to Proposition~\ref{prop:bijection}(iii)), and thus the same set of keys in $T_n^{(1)}$ as in $T^{(1)}$. Moreover, by Lemma~\ref{lemma:Hfact}, these keys are moved upwards from the leaves in the relative order described by $\pi^{(1)}$, hence $T_n^{(1)}=T^{(1)}$ and $H$ is a history of $T$ that belongs to $\hat\Psi^{-1}(\pi^{(1)})$.  
 
 Conversely, consider now a history $(T_1,\dots,T_n)\in\hat\Psi^{-1}(\pi^{(1)})$ with the associated historic tree $H_n$. As in step 1, we obtain from $\pi^{(1)}$ and $T_n$ a sequence $\pi_\iota=(K_{i_1},\dots,K_{i_{n_1}})$, where the $i_j$ are precisely the times in the history when a leaf was split. The historic tree $H_n$ therefore must have its branchings labelled by the $i_j$ (and this forces the labelling to be increasing along the red edges in $G$), and to be consistent with Lemma~\ref{lemma:Hfact}, the left/right-positioning of the branchings has to correspond to the one in a binary search tree obtained from $\pi_\iota$ or equivalently $\pi^{(1)}$. Thus, such an $H_n$ is of the type constructed in Step 2, and its labelling is a topological labelling of $G$.
\end{proof}

\begin{proof}[Proof of Lemma~\ref{lemma:step3}]
 Claim (i) is evident from the construction. For (ii), the equality $|\mc H^\pm|=|\mc R^\pm|$ follows from Lemma~\ref{lemma:Hfact}, we have $|\mc R^\pm|\geq m$ since $|\mc R|\geq 2m+1$, and consecutivity is again immediate from the construction. Claim (iii) again follows from Lemma~\ref{lemma:Hfact}. Taken together, these three claims ensure that the recursion in step 3 is well-defined whenever we initialize as in step 3.0. 
 
 For the final assertion, we first define $\pi^\pm$ to be permutations obtained from $p^\pm$ in step 3.2 by mapping $\mc R^\pm$ to $\{1,\dots,|\mc R^\pm|\}$ in an order-preserving fashion. We now use strong induction on $n$, where $T, \pi^{(1)}$, and $H$ are arbitrary but coherent in the sense of Lemma~\ref{lemma:digraph}. For $n\leq 2m$, there is nothing to show, as step 3.1 gives $\permset(H)=S_n$. 
 
 For all larger $n$, observe that if $\pi\in\permset(H)$ then $K_{i_1}$ is the median of $\pi(1),\dots,\pi(2m+1)$ and $\pi^\pm\in \permset(\mc H^\pm)$. The first property is equivalent to $K_{i_1}$ being moved upwards from the leaves at the first branching of $H$ and is ensured by step 3.1. The second property comes from Lemma~\ref{lemma:Hfact}: The large entries in $\pi$ are precisely (except for the first $m$) those corresponding to the right descendants of the first branching in $H$. This holds by step 3.2, and $\pi^\pm$ are constructible by the induction hypothesis since $|\mc R^\pm|<n$. Thus, such $\pi$ is constructible. 
 
 Conversely, suppose $\pi$ is constructible. Then, by the recursion, $\pi^\pm$ are constructible, and thus $\pi^\pm\in\permset(\mc H^\pm)$ by the induction hypothesis. The $i$-th entry of $\pi^\pm$ is simultaneously the $i$-th large/small entry (according to step 3.1) of $\pi$. For $i<m$, this is chosen arbitrarily among all possible configurations for the first $2m+1$ entries of $\pi$. For $i>m$, the entry in $\pi$ is dictated by $H$, but is not among the first $2m+1$. Thus, it corresponds to a descendant of the first branching, and it follows from Lemma~\ref{lemma:Hfact} that filling the large/small entries with the entries from $\pi^\pm$ in order produces a $\pi\in\permset(H)$.  
\end{proof}

\begin{proof}[Proof of Proposition~\ref{prop:permnumber}]
 This lemma follows from an analysis of step 3. Indeed, whenever we are placing a $\mc K_1$ into the permutation in step 3.1, we have $2m+1$ choices for the exact position, and then $\binom{2m}{m}$ choices for the location of the small positions within the first $2m+1$ slots of $p$. Placing such a $\mc K_1$ corresponds exactly to the branchings in $H$, and it is clear that every possible choice will lead to a different $\pi$ in the end. Whenever we have $|\mc R|\leq 2m$, we have $|\mc R|!$ choices. Moreover, invoking Lemma~\ref{lemma:step3}, the corresponding $\mc K$ is the empty sequence, hence the entries of $\mc R$ are the keys that end up in a joint leaf---say, the $i$-th leaf---of $T$. Thus by Proposition~\ref{prop:bijection}(iii), $|\mc R|=m+s_i$, and \eqref{eq:permnumber} follows. 
\end{proof}

\section{The number of histories}

In this section, we will be interested in the number of possible histories that can arise, and in particular the asymptotic behaviour of this number. We focus on the case $m=1$. In this case, a historic tree is a binary increasing tree where only vertices at even heights can have two children. Vertex $2$ is always the only child of vertex $1$, and vertex $3$ is always the only child of vertex $2$. It will be advantageous later to remove vertex $1$ (and decrease all other labels by $1$); we call the result a \emph{reduced historic tree}. In such a tree, only vertices at odd heights can have two children.

If we remove the top two vertices (vertices $1$ and $2$, referred to in the following as the \emph{stem}) from a reduced historic tree with $n$ vertices, then it decomposes into two smaller reduced historic trees, each possibly only consisting of a single external vertex.
On the level of generating functions, this translates to a second-order differential equation for the exponential generating function $H(x) = \sum_{n \geq 0} \frac{h_n}{n!} x^n$, where $h_n$ is the number of reduced historic trees with $n$ internal vertices (equivalently, the number of histories of length $n+1$). We have
\begin{equation}\label{eq:diffeq2}
 H''(x) = H(x)^2,\qquad H(0) = H'(0) = 1.
\end{equation}
One can compare this to the well-known differential equation $T'(x) = T(x)^2$ for the exponential generating function associated with arbitrary binary increasing trees, see for example \cite[Lemma 6.4]{drmota09}. We remark here that the tree consisting only of a single external vertex is often not counted, in which case the equation becomes $T'(x) = (1+T(x))^2$ instead. The sequence $h_n$ (see \cite[A007558]{OEIS}) and the associated differential equation~\eqref{eq:diffeq2} were analysed in a different context in \cite{bodini16}: the differential equation has an explicit solution that can be expressed in terms of the Weierstrass elliptic function. It has a dominant singularity at $\rho \approx 2.3758705509$ where
\begin{equation*}
H(x) \sim \frac{c}{(1-x/\rho)^2}
\end{equation*}
for a constant $c = 6\rho^{-2} \approx 1.0629325375$. This leads to the following asymptotic behaviour:
\begin{equation*}
\frac{h_n}{n!} \sim cn \rho^{-n} = 6n \rho^{-n-2}.
\end{equation*}

We would now like to generalise the differential equation to arbitrary $m \geq 1$. We remove the first $m$ vertices from a $(2m+1)$-historic tree to obtain a reduced $(2m+1)$-historic tree, which is now a binary increasing tree where only vertices at heights $\equiv -1 \mod m+1$ can have two children. Removing the stem consisting of $m+1$ vertices decomposes such a tree into two smaller trees with the same property (each of them can also be a single external vertex). So in analogy to~\eqref{eq:diffeq2}, we obtain a differential equation of order $m+1$, namely
\begin{equation}\label{eq:diffeqm}
 H^{(m+1)}(x) = H(x)^2,\qquad H(0) = H'(0) = \cdots = H^{(m)}(0) = 1.
\end{equation}

This higher-order differential equation can no longer be solved in an explicit fashion, as it was the case for $m=1$. If we assume that there is a dominant singularity $\rho_m$ where the behaviour of $H$ is of the form $c_m (1-x/\rho_m)^{-a_m}$, then comparing the two sides of the equation gives us
\begin{equation*}
c_m \frac{a_m(a_m+1)\cdots(a_m+m)}{\rho_m^{m+1}}
(1-x/\rho_m)^{-a_m-m-1} = c_m^2 (1-x/\rho_m)^{-2a_m},
\end{equation*}
thus $a_m = m+1$ and $c_m = \frac{(2m+1)!}{m!} \rho_m^{-m-1}$. Applying singularity analysis would then yield
\begin{equation*}
[x^n] H(x) \sim \frac{c_m}{m!} n^m \rho_m^{-n} = \frac{(2m+1)!}{(m!)^2} n^m \rho_m^{-n-m-1}.
\end{equation*}
This leads us to the following conjecture:

\begin{conjecture}\label{conj:general_m}
For every $m \geq 1$, the number of reduced $(2m+1)$-historic trees with $n$ vertices (corresponding to histories of length $n+m$) is asymptotically equal to
\begin{equation*}
n! \cdot \frac{(2m+1)!}{(m!)^2} n^m \rho_m^{-n-m-1}
\end{equation*}
for some positive constant $\rho_m$.
\end{conjecture}

Numerical evidence for small values of $m$ seems to support this conjecture, as the fit of the asymptotic formula with the actual coefficients is excellent. Experimental values of the exponential growth rate $\rho_m^{-1}$ are given in Table~\ref{tab:num-values}.

\begin{table}[htbp]
\centering
\begin{tabular}{|c|ccccc|}
\hline
$m$ &2 & 3 & 4 & 5 & 6 \\
\hline
$\rho_m^{-1}$ & $3.7746$ & $5.1792$ & $6.5857$ & $7.9928$ & $9.3999$ \\
\hline
\end{tabular}
\caption{Experimental values of $\rho_m^{-1}$ for $2 \leq m \leq 6$.}\label{tab:num-values}
\end{table}

\section{Statistics of \textit{B}-trees via historic trees}

Let us now study $B$-trees that are constructed by successive insertion of $n$ random numbers. Equivalently, we can think of them as being constructed from a random permutation of $1,2,\ldots,n$. In order to apply the connection to historic trees, we need to take the number of permutations associated with a specific history into account.

Again, we focus on the special case $m=1$. Proposition~\ref{prop:permnumber} tells us that the number of permutations corresponding to a specific historic tree $T$ is in this case $6^{\br(T)}2^{\betw(T)}$, where $\br(T)$ is the number of branchings and $\betw(T)$ the number of internal vertices that lie directly between a branching and an external vertex. This remains true if we consider reduced historic trees. We associate this number as a weight $w(T) = 6^{\br(T)}2^{\betw(T)}$ with every reduced historic tree $T$ and consider the weighted exponential generating function (rather than the unweighted one that was analysed in the previous section). In the recursive decomposition of a reduced historic tree into its stem and two smaller trees $T_1$ and $T_2$, we have $\br(T) = \br(T_1) + \br(T_2) + 1$ and $\betw(T) = \betw(T_1) + \betw(T_2)$. This is even true if $T_1$ or $T_2$ (or both) only consist of a single external vertex. Thus we obtain
\begin{equation}\label{eq:weight_rec}
w(T) = 6^{\br(T_1) + \br(T_2)+1} \cdot 2^{\betw(T_1) + \betw(T_2)} = 6 w(T_1) w(T_2).
\end{equation}
On the level of the weighted exponential generating function $W(x)$,~\eqref{eq:diffeq2} becomes
\begin{equation}\label{eq:diffeqw2}
 W''(x) = 6W(x)^2,\qquad W(0) = 1, W'(0) = 2.
\end{equation}
Unlike~\eqref{eq:diffeq2}, however, there is now a very simple explicit solution, namely $W(x) = \frac{1}{(1-x)^2}$. This is not unexpected, since the total weight of all reduced historic trees with $n$ internal vertices must be equal to the number of permutations of $1,2,\ldots,n+1$. Thus 
\begin{equation*}
W(x) = \sum_{n \geq 0} \frac{(n+1)!}{n!} x^n = \frac{1}{(1-x)^2}.
\end{equation*}

Recall that the number of external vertices in $n$-vertex reduced $3$-historic trees is in bijection with the number of leaves in $2$-$3$-trees built from $n+1$ keys. Thus, as a next step, we incorporate the number of external vertices $\ext(T)$ as an additional statistic in our generating function in order to prove the following theorem:

\begin{theorem}\label{thm:leaf-distribution}
Let $L_n$ be the number of leaves in a $2$-$3$-tree built from $n$ random keys. Then we have $\Ex(L_n) = \frac37 (n+1)$ and $\Var(L_n) = \frac{12}{637} (n+1)$ for $n > 11$. Moreover, the central limit theorem
\begin{equation*}
\frac{L_n - \Ex(L_n)}{\sqrt{\Var(L_n)}} \overset{\mathrm{d}}{\to} N(0,1)
\end{equation*}
holds.
\end{theorem}

\begin{proof}
Let us consider the bivariate generating function in which the second variable $u$ marks the number of external vertices $\ext(T)$:
\begin{equation*}
W(x,u) = \sum_T \frac{1}{|T|!} x^{|T|} u^{\ext(T)}.
\end{equation*}
Since $\ext(T) = \ext(T_1) + \ext(T_2)$, the differential equation~\eqref{eq:diffeqw2} is actually unaffected by the additional variable; the only change concerns the initial values. We have (where derivatives are taken with respect to $x$)
\begin{equation}\label{eq:diffeqwu2}
 W''(x,u) = 6W(x,u)^2,\qquad W(0,u) = u, W'(0,u) = 2u,
\end{equation}
which no longer has an equally simple explicit solution. Using the method described in~\cite{bodini16}, it can, however, be expressed as the inverse function to
\begin{equation*}
X(w,u) = \int_u^w \frac{1}{\sqrt{4t^3+4u^2(1-u)}} \,dt.
\end{equation*}
Hence $W(x,u)$ has a dominant singularity at $\rho(u) = \int_u^{\infty} \frac{1}{\sqrt{4t^3+4u^2(1-u)}} \,dt$: as $w \to \infty$, we have
\begin{equation*}
X(w,u) = \rho(u) - \frac{1}{\sqrt{w}} + O(w^{-7/2}),
\end{equation*}
thus
\begin{equation*}
W(x,u) \sim \frac{1}{(\rho(u)-x)^2}
\end{equation*}
at the singularity. An application of the quasi-power theorem \cite[Theorem IX.8]{FS2009} yields a central limit theorem for the number of external vertices. Moreover, one can obtain explicit expressions for the moments. Differentiating~\eqref{eq:diffeqwu2} with respect to $u$ and plugging in $u=1$, we obtain the following differential equation for $W_1(x) = \frac{\partial}{\partial u} W(x,u)\Big|_{u=1}$:
\begin{equation*}
W_1''(x) = 12W(x,1)W_1(x) = \frac{12}{(1-x)^2} W_1(x),\qquad W_1(0) = 1, W_1'(0) = 2,
\end{equation*}
since we already know that $W(x,1) = W(x) = (1-x)^{-2}$. This linear differential equation  has the two linearly independent solutions $(1-x)^{-3}$ and $(1-x)^4$, and one obtains
\begin{equation*}
W_1(x) = \frac{6}{7(1-x)^3} + \frac{(1-x)^4}{7}.
\end{equation*}
Thus for $n > 4$, we have $[x^n] W_1(x) = \frac67 [x^n] (1-x)^{-3} = \frac67 \binom{n+2}{2}$. Consequently, the average number of external vertices is
\begin{equation*}
  \frac{[x^n] W_1(x)}{[x^n] W(x)} = \frac{\frac67 \binom{n+2}{2}}{n+1} = \frac{3(n+2)}{7}.
\end{equation*}
In the same way, one can treat the second moment: to this end, we consider $W_2(x) = \big( \frac{\partial}{\partial u} \big)^2 W(x,u)\Big|_{u=1}$. Differentiating~\eqref{eq:diffeqwu2} twice with respect to $u$ and plugging in $u=1$ now gives us
\begin{align*}
W_2''(x) &= 12W(x)W_2(x) +12 W_1(x)^2 \\
&= \frac{12}{(1-x)^2} W_2(x) + 12 \Big( \frac{6}{7(1-x)^3} + \frac{(1-x)^4}{7} \Big)^2,
\end{align*}
and $W_2(0) = W_2'(0) = 0$. 
The solution to this differential equation is given by
\begin{equation*}
W_2(x) = \frac{54}{49 (1-x)^4} -\frac{108}{91 (1-x)^3} - \frac{24}{49}
   (1-x)^3+\frac{4}{7} (1-x)^4+\frac{2}{637} (1-x)^{10}.
\end{equation*}
So for $n > 10$, $[x^n] W_2(x) = \frac{54}{49} \binom{n+3}{3} - \frac{108}{91} \binom{n+2}{2} = \frac{9(n+1)(n+2)(13n-3)}{637}$. It follows that the variance of the number of external vertices is
\begin{equation*}
\frac{[x^n] (W_1(x)+W_2(x))}{[x^n] W(x)} - \Big( \frac{[x^n] W_1(x)}{[x^n] W(x)} \Big)^2 = \frac{12(n+2)}{637}.
\end{equation*}
This completes the proof.
\end{proof}

The approach in our proof provides an alternative to the analysis via P\'olya urns, see \cite{AFP88,BP85,BY95,yao78} (in particular, the mean was first determined by Yao \cite{yao78} by explicitly solving a recursion). Here, one can think of the leaves in a $B$-tree as balls in an urn of different types depending on the number of keys they hold. Adding a new key then corresponds to picking a ball from the urn and replacing it by a new ball (of different type), or two new balls in the case of a node split.

The same calculations for the moments as in Theorem~\ref{thm:leaf-distribution} can also be carried out for higher values of $m$, though the expressions become more complicated. For general $m \geq 1$, the differential equation becomes
\begin{equation*}
 W^{(m+1)}(x,u) = \frac{(2m+1)!}{m!^2} W(x,u)^2,
\end{equation*}
with initial values
\begin{equation*}
W^{(i)}(0,u) = (m+i)!u,\qquad i=0,1,\ldots,m.
\end{equation*}
In particular, we have $W(x) = W(x,1) = m!(1-x)^{-m-1}$, and $W_1(x) = \frac{\partial}{\partial u} W(x,u)\Big|_{u=1}$ satisfies the linear differential equation
\begin{equation*}
W_1^{(m+1)}(x) = \frac{2(2m+1)!}{m!} (1-x)^{-m-1}W_1(x).
\end{equation*}
Up to a trivial change of variables (substituting for $1-x$), this is a linear differential equation of Cauchy--Euler type that can be solved with standard tools. In fact, setting $e^{-t} = 1-x$ turns it into a linear differential equation with constant coefficients. Functions of the form $f(x) = (1-x)^{-b}$ with 
\begin{equation}\label{eq:b-eq}
b^{\overline{m+1}} = b(b+1)\cdots (b+m) = \frac{2(2m+1)!}{m!} = \frac{(2m+2)!}{(m+1)!}
\end{equation}
are particular solutions to this differential equation. Note that $b=m+2$ is always a solution to~\eqref{eq:b-eq}. The general solution can be determined as linear combination of particular solutions, taking the initial values into account. The term of the form $c (1-x)^{-m-2}$ in $W_1(x)$ dominates asymptotically.

To give one more concrete example, for $m=2$ we have the differential equation
\begin{equation*}
 W'''(x,u) = 30W(x,u)^2,\qquad W(0,u) = 2u, W'(0,u) = 6u, W''(0,u) = 24u.
\end{equation*}
Thus $W(x,1) = 2(1-x)^{-3}$. The solutions to~\eqref{eq:b-eq} are now $b=4$ and $b = \frac{-7 \pm \sqrt{71} i}{2}$. Taking the initial values into account, we obtain
\begin{multline*}
W_1(x) = \frac{60}{37(1-x)^4} + \frac{7\sqrt{71} +31i}{37\sqrt{71}} (1-x)^{(7+\sqrt{71}i)/2}\\
+ \frac{7\sqrt{71} -31i}{37\sqrt{71}} (1-x)^{(7-\sqrt{71}i)/2}.
\end{multline*}
Thus the average number of external vertices in reduced $5$-historic trees with $n$ vertices is $\frac{10(n+3)}{37} + O(n^{-13/2})$.

For general $m$, one finds from the differential equation that the function $\ell(t) = W_1(1-e^{-t})$ has Laplace transform
\begin{equation*}
L(s) = \frac{m! ((m+1)^{\overline{m+1}} - s^{\overline{m+1}})}{(s-m-1)((m+2)^{\overline{m+1}} - s^{\overline{m+1}})},
\end{equation*}
where $s^{\overline{h}} = s(s+1)\cdots(s+h-1)$ is a rising factorial as in~\eqref{eq:b-eq}. The term $\frac{\kappa_m}{s-m-2}$ in the partial fraction decomposition corresponds to the dominant term
\begin{equation*}
 W_1(x) \sim \frac{\kappa_m}{(1-x)^{m+2}}.
\end{equation*}
Here we have, with $H_k = 1 + \frac12 + \cdots + \frac{1}{k}$ denoting a harmonic number,
\begin{equation*}
\kappa_m = \frac{m!}{2 (H_{2 m + 2} - H_{m+1})}.
\end{equation*}
Consequently, the average number of external vertices in reduced $(2m+1)$-historic trees with $n$ vertices is asymptotically equal to $\frac{\kappa_m}{(m+1)!} \cdot n = \frac{1}{2(m+1)(H_{2 m + 2} - H_{m+1})} \cdot n$. Some explicit values of the constant $\frac{\kappa_m}{(m+1)!}$ are given in Table~\ref{tab:const-values}.

\begin{table}[htbp]
\centering
\begin{tabular}{|c|cccccccccc|}
\hline
$m$ &1 &2 & 3 & 4 & 5 & 6 & 7 & 8 & 9 & 10 \\
\hline
$\frac{\kappa_m}{(m+1)!}$ & $\frac37$ & $\frac{10}{37}$ & $\frac{105}{533}$ & $\frac{252}{1627}$ & $\frac{2310}{18107}$ & $\frac{25740}{237371}$ & $\frac{9009}{95549}$ & $\frac{136136}{1632341}$ & $\frac{11639628}{155685007}$ & $\frac{10581480}{156188887}$ \\
\hline
\end{tabular}
\caption{Values of $\frac{\kappa_m}{(m+1)!} = \frac{1}{2(m+1)(H_{2 m + 2} - H_{m+1})} $ for $1 \leq m \leq 10$.}\label{tab:const-values}
\end{table}

\section{Conclusion and perspective}

The connection between $B$-tree histories and historic trees provides us with a novel way to analyse $B$-trees and their evolution. Possible future directions include studying further statistics of $B$-trees and historic trees and considering higher values of $m$. In particular, a proof of Conjecture~\ref{conj:general_m} would be desirable. It might even be interesting, at least from a purely mathematical perspective, to allow $m$ to grow with $n$.

\section*{Acknowledgements}

F. Burghart has received funding from the European Union's Horizon 2020 research and innovation programme under the Marie Sk\l{}odowska-Curie grant agreement no. 101034253. S. Wagner was supported by the Swedish research council (VR), grant 2022-04030.

\bibliographystyle{alphaurl}
\bibliography{BtreesBib}

\end{document}